\documentclass[reqno,12pt]{amsart}
\usepackage{amsmath,amsthm,amsfonts,amssymb}
\usepackage{enumerate}

\newtheorem{theorem}{Theorem}[section]

\newtheorem{lemma}[theorem]{Lemma}
\newtheorem{corollary}[theorem]{Corollary}
\newtheorem{proposition}[theorem]{Proposition}

\begin{document}

\centerline{{\bf Some remarks on Davie's uniqueness theorem}}

\vskip .2in

\centerline{{\bf
A.V.~Shaposhnikov\footnote{Department
of Mechanics and Mathematics, Moscow State University, Moscow, Russia, e-mail: shal1t7@mail.ru}}}

\vskip .2in
\centerline{{\bf Abstract}}

\vskip .1in

We present a new approach to Davie's theorem on the uniqueness of solutions to the equation $dX_t = b(t, X_t)\,dt + dW_t$
for almost all Brownian paths.
A generalization of this result and a discussion of some close problems are given.

AMS Subject Classification: 60H10, 34F05, 46N20.

Keywords: Brownian motion, stochastic differential equation, pathwise uniqueness

\vskip .2in

\section{Introduction}

In this paper we consider the stochastic differential equation
\begin{equation}\label{eq:main}
X_{t} = x + W_{t} + \int_{0}^{t}b(s, X_s)\,ds.
\end{equation}
In the paper \cite{D} A.M. Davie proved the following theorem:
\begin{theorem}\label{th:borel_case}
Let $b: [0, T]\times \mathbb{R}^{d}\mapsto\mathbb{R}^{d}$ be a Borel measurable bounded mapping.
Then for almost all Brownian paths the equation \ref{eq:main} has exactly one solution.
\end{theorem}

The proof of Davie is quite self-contained, but rather technically complicated.
In particular it does not rely on the uniqueness of strong solutions.
It turns out that in some cases the pathwise uniqueness can be proved with a slightly simpler approach.
The main idea is to use the H\"older regularity of the flow generated
by the strong solution proved in \cite{FF} and a modification
of the van Kampen uniqueness theorem for ordinary differential equations with a Lipschitz
flow and continuous coefficients (see \cite{VK}).
This approach also enables us to extend Davie's result to some other classes of irregular drifts.

\section{Auxiliary results}

The proof of Davie uses the following estimate:

\begin{proposition}\label{pr:estimate}
Let $b \in C\bigl([0, 1], C^{1}_{b}(\mathbb{R}^{d}, \mathbb{R}^{d})\bigr), \ \|b\|_{\infty} \leqslant 1$. There exist positive constants
$C,\alpha$ (that do not depend on $b$) such that the following inequality holds:
$$
\mathbb{E}\exp\alpha\Bigl|\int_{0}^{1}b'_{x}(t, W_t)\,dt\Bigr|^{2} \leqslant C.
$$
\end{proposition}

An interesting discussion of this inequality and some similar problems can be found in \cite{F}.
The original proof of Davie is quite long and relies on some explicit computations for the Gaussian kernel.
Since our approach to Davie's theorem in the case of a Borel measurable drift also uses this estimate,
below we present a proof which seems to be less technical than the reasoning in \cite{D}.

\begin{proof}
We first prove the desired inequality for $d = 1$.
Let
$$Z_s := b(s, W_s).$$
For the quadratic covariation of the processes $Z$ and $W$
we have the following representations (see \cite{FPS}):
$$
[Z, W]_{1} = \int_{0}^{1}b'_{x}(s, W_s)\,ds
$$
$$
[Z, W]_{1} = \lim \sum (Z_{t_{i + 1}} - Z_{t_{i}})(W_{t_{i + 1}} - W_{t_{i}})
$$
$$
[Z, W]_{1} = \int_{0}^{1}Z_{t}\,d^{*}W_t - \int_{0}^{1}Z_t\,dW_t,
$$
where
$$
\int_{0}^{1}Z_{t}\,d^{*}W_t = \int_{0}^{1}Z_{1 - s}\,d\widetilde{W}_{s}, \ \widetilde{W}_{s} = W_{1 - s}.
$$
The process $\widetilde{W}_{t}$ (the time-reversed Brownian motion) satisfies the integral equality
$$
\widetilde{W}_{t} = \widetilde{W}_{0} + B_t + \int_{0}^{t}\frac{-\widetilde{W}_{s}}{1 - s}\,ds,
$$
where $B$ is another Brownian motion. Then
\begin{multline*}
\int_{0}^{1}b'_{x}(t, W_t)\,dt = \\
= \int_{0}^{1}b(1 - t, W_{1 - t})\,dB_t + \int_{0}^{1}\frac{-W_{1 - t}b(1 - t, W_{1 - t})}{1 - t}\,dt - \int_{0}^{1}b(t, W_t)\,dW_t = \\
= I_1 + I_2 + I_3 \\
\end{multline*}
It is easy to notice that the terms $I_1$ and $I_3$ can be estimated by means of
the Dubins--Schwarz theorem and the well-known formula for the distribution
of the maximum of a Wiener process on the interval $[0, 1]$.
The assumption that $\|b\|_{\infty} \leqslant 1$ implies that there exist constants
$\alpha_1, C_1 > 0$ such that
$$
\mathbb{E}\exp\alpha_{1}\Bigl(I_{1}^{2} + I_{3}^{2}\Bigr) \leqslant C_{1}.
$$
Let us estimate the term $I_2$.
Applying Jensen's inequality we obtain the following estimates:
\begin{multline*}
\mathbb{E}\exp{\frac{1}{16}I_{2}^{2}} = \mathbb{E}\exp{\frac{1}{4}\Biggl(\int_{0}^{1}b(1 - t, W_{1 - t})\frac{W_{1 - t}}{2 - 2t}\, dt\Biggr)^2} \leqslant \\
\leqslant \mathbb{E}\int_{0}^{1}\exp{\Biggl(\frac{1}{4}b^{2}(1 - t, W_{1 - t})\Bigl|\frac{W_{1 - t}}{\sqrt{1 - t}}\Bigr|^{2}\Biggr)}\frac{dt}{2\sqrt{1 - t}} \leqslant \\
\leqslant \mathbb{E}\int_{0}^{1}\exp{\Biggl(\frac{1}{4}\Bigl|\frac{W_{1 - t}}{\sqrt{1 - t}}\Bigr|^{2}\Biggr)}\frac{ds}{2\sqrt{1 - t}} \leqslant C_2 < \infty.
\end{multline*}
Now it is trivial to complete the proof in the case $d = 1$.

Now let $d > 1$. We have
$$
b(t, x) = \bigl(b^{1}(t, x_1, \ldots, x_d), \ldots, b^{d}(t, x_1, \ldots, x_d)\bigr),
$$
$$W_t = \bigl(W^{1}_t, \ldots, W^{d}_{t}\bigr).$$
It is easy to see that in this case it suffices  to prove the inequality
$$
\mathbb{E}\exp{\alpha\Bigl|\int_{0}^{1}b'_{x_1}(t, W^{1}_{t}, \ldots, W^{d}_{t})\, dt\Bigr|^2} \leqslant C
$$
for all functions $b$ with $\|b\|_{\infty} \leqslant 1$.
This estimate follows from the chain of inequalities:
\begin{multline*}
\mathbb{E}\exp{\alpha\Bigl|\int_{0}^{1}b'_{x_1}(t, W^{1}_{t}, \ldots, W^{d}_{t})\, dt\Bigr|^{2}} =
\\
= \mathbb{E}\Bigl[\mathbb{E}\Bigl[\exp{\alpha\Bigl|\int_{0}^{1}b'_{x_1}(t, W^{1}_{t}, \ldots, W^{d}_{t})\, dt\Bigr|^{2}} | W^{2}, \ldots, W^{n}\Bigr]\Bigr]
\leqslant \mathbb{E}C = C,
\end{multline*}
where the one-dimensional case has been used.
\end{proof}

\begin{corollary}\label{cor:refined_estimate}
There exist constants $C, \alpha >0$ such that, for any Borel
measurable mapping $b \in L^{\infty}\bigl([r, u]\times\mathbb{R}^{d}, \mathbb{R}^{d}\bigr)$
with $\|b\|_{\infty} \leqslant 1$,
any Borel measurable functions $h_1, h_2 \in L^{\infty}\bigl([r, u], \mathbb{R}^{d}\bigr)$
and any $\lambda \geqslant 0$, the following inequality holds:
$$
P\Bigl[\bigl|\int_{r}^{u}b(s, W_s + h_1(s)) - b(s, W_s + h_2(s))\, ds\bigr| \geqslant \lambda l^{\frac{1}{2}}\|h_1 - h_2\|_{\infty}\Bigr]
\leqslant C\exp\bigl(-\alpha \lambda^2\bigr),
$$
where $l = u - r$.
\end{corollary}
\begin{proof}
Taking into account the scale invariance of the Brownian motion it is easy to notice
that we can assume that $r = 0$ and $u = 1$.
One can easily show that it is also sufficient to prove the desired estimate
just for smooth functions with compact supports.
In this case we have
\begin{multline*}
\mathbb{E}\exp{\alpha\Bigl|\int_{0}^{1}\frac{b(s, W_s + h_1(s)) - b(s, W_s + h_2(s))}{\|h_1 - h_2\|_{\infty}}\, ds\Bigr|^2} = \\
= \mathbb{E}\exp{\alpha\Bigl|\int_{0}^{1}\int_{0}^{1}b'_{x}\bigl(s, W_s + h_2(s) + \theta(h_1(s) - h_2(s))\bigr)\frac{h_1(s) - h_2(s)}{\|h_1 - h_2\|_{\infty}}\, d\theta\, ds\Bigr|^2}
\leqslant \\
\leqslant \int_{0}^{1}\mathbb{E}\exp{\alpha\Bigl|\int_{0}^{1}b'_{x}\bigl(s, W_s + h_2(s) + \theta(h_1(s) - h_2(s))\bigr)\frac{h_1(s) - h_2(s)}{\|h_1 - h_2\|_{\infty}}\, ds\Bigr|^2}\, d\theta \leqslant \\
\leqslant \int_{0}^{1}C\, d\theta = C.
\end{multline*}
In the last inequality for each $\theta$ we have applied Proposition \ref{pr:estimate} to the function
$$
\widehat{b}(s, x) = b\bigl(s, x + h_2(s) + \theta(h_1(s) - h_2(s))\bigr)\frac{h_1(s) - h_2(s)}{\|h_1 - h_2\|_{\infty}}.
$$
Now the necessary estimate follows by the Chebyshev inequality.
\end{proof}

The next proposition will play the crucial role in the proof of the main results.

\begin{proposition}\label{pr:flow}
Let
$$
b \in L^{q}\bigl([0, T], L^{p}(\mathbb{R}^{d})\bigr), \ \frac{d}{p} + \frac{2}{q} < 1.
$$
Then, there exists a H\"older flow of solutions to the equation \ref{eq:main}. More precisely,
for any filtered probability space
$(\Omega, \mathcal{F}, \{\mathcal{F}_{t}\}, P)$ and a Brownian motion $W$, there exists a mapping
$(s, t, x, \omega) \mapsto \varphi_{s, t}(x)(\omega)$ with values in $\mathbb{R}^{d}$, defined for
$0 \leqslant s \leqslant t \leqslant T, \ x \in \mathbb{R}^{d}, \ \omega \in \Omega$,
such that for each $s \in [0, T]$ the following conditions hold:
\begin{enumerate}[{1.}]
\item for any $x \in \mathbb{R}^{d}$ the process $X_{s, t}^{x} = \varphi_{s, t}(x)$ is a continuous $\mathcal{F}_{s,t}$ adapted
solution to the equation \ref{eq:main},
\item $P$-almost surely the mapping $x \mapsto \varphi_{s, t}(x)$ is a homeomorphism,
\item $P$-almost surely for all $x \in \mathbb{R}^{d}$ and $0 \leqslant s \leqslant u \leqslant t \leqslant 1$
$$\varphi_{s, t}(x) = \varphi_{u, t}(\varphi_{s, u}(x)),$$
\item $P$-almost surely for each $\alpha \in (0, 1)$ and each positive $N \in \mathbb{R}$
one can find $C(\alpha, N, \omega) < \infty$ such that for all
$x, y \in \mathbb{R}^{d}: |x|, |y| < N$ and
$s, t \in [0, T], \ s \leqslant t$
$$|\varphi_{s, t}(x) - \varphi_{s, t}(y)| \leqslant C(\alpha, T, N, \omega)|x - y|^{\alpha}.$$
\end{enumerate}
\end{proposition}

The existence of a flow possessing properties 1--3 is proved in \cite{FF} (see Theorem 1.2).
Instead of property 4 the authors of \cite{FF} prove (see their Lemma 5.11)
a slightly weaker assertion that almost surely for any fixed $s, t \in [0, 1], \ s \leqslant t$
the mapping $\varphi_{s, t}$ is H\"older continuous.
For the sake of completeness,   we present below
a sketch of the proof of Proposition \ref{pr:flow} with necessary references
to \cite{FF}, \cite{FF2} and the key details of the proof of property~4.

\vskip .1in
{\bf Step 1.} (See \cite{FF}, Theorem 3.3, Lemma 3.4, and Lemma 3.5.)
Let
$$
L_{p}^{q}(T) = L^{q}\bigl([0, T], L^{p}(\mathbb{R}^{d})\bigr)
$$
$$
\mathbb{H}_{\alpha, p}^{q}(T) = L^{q}\bigl([0, T], W^{\alpha, p}(\mathbb{R}^{d})\bigr),
\ \mathbb{H}_{p}^{\beta, q}(T) = W^{\beta, q}\bigl([0, T], L^{p}(\mathbb{R}^{d})\bigr)
$$
$$
H_{\alpha, p}^{q}(T) = \mathbb{H}_{\alpha, p}^{q}(T) \cap \mathbb{H}_{p}^{1, q}(T)
$$
Let $U:[0, T]\times\mathbb{R}^{d}\rightarrow\mathbb{R}^{d}$ be a solution to the equation
\begin{equation}
\label{eq:pde}
\left\{
\begin{aligned}
&\frac{\partial U}{\partial t} + \frac{1}{2}\Delta U + b \cdot \nabla U = \lambda U - b\\
&U(T, x) = 0
\end{aligned}
\right.
\end{equation}
for sufficiently large positive $\lambda$ such that
$$
\|U\|_{H_{2, p}^{q}(T)} = \|D_{t}U\|_{L_{p}^{q}} + \|U\|_{H_{2, p}^{q}(T)}
\leqslant C(d, T, p, q, \lambda)\|b\|_{L_{p}^{q}(T)},
$$
$$
\sup\limits_{t \in [0, T]}\|\nabla U\|_{C_{b}(\mathbb{R}^{d})} \leqslant \frac{1}{2}.
$$
Then the family of mappings $\psi_{t}:\mathbb{R}^{d}\rightarrow\mathbb{R}^{d}$ defined by the formula
$$
\psi_{t}(x) = x + U(t, x)
$$
possesses the following properties:
\begin{enumerate}[{1.}]
\item for each $t \in [0, T]$ the mapping $\psi_{t},\ \psi_{t}^{-1}$ is a
 $C^{1}$-diffeomorphism of $\mathbb{R}^{d}$,
\item uniformly in $t \in [0, T]$ the mappings $\psi_{t},\ \psi_{t}^{-1}$,
have globally bounded H\"older-continuous derivatives with respect to the space variable,
\item the mapping $(t, x) \mapsto \psi_{t}(x)$  belongs locally to the class $H_{2, p}^{q}(T)$.
\end{enumerate}

{\bf Step 2.} (See \cite{FF}, Proposition 4.3.)
The next step is transforming the original equation \ref{eq:main}
(considered as a stochastic equation with the identity diffusion matrix and a Borel measurable drift)
into an equation with more regular coefficients by means of the family of
the homeomorphisms constructed at the previous step.
Let us apply It\^o's formula to the process $X_t$ and the function $U$ (see \cite{FF}, p.~4):
\begin{multline*}
dU(t, X_t) = \frac{\partial U}{\partial t}(t, X_t)\,dt +
\nabla U(t, X_t)\bigl(b(t, X_t)\,dt + dW_t\bigr) +
\frac{1}{2}\Delta U(t, X_t)\,dt=\\
=\lambda U(t, X_t) - b(t, X_t)dt + \nabla U(t, X_t)\,dW_t\\
\end{multline*}
Then the process
$$
Y_t := \psi_t(t, X_t) = X_t + U(t, X_t)
$$
has the stochastic differential
\begin{multline*}
dY_t =
\lambda U(t, \psi_{t}^{-1}(Y_t))\,dt + \bigl[I + \nabla U(t, \psi_{t}^{-1}(Y_t))\bigr]\,dW_t=
\widetilde{b}(t, Y_t)\,dt + \widetilde{\sigma}(t, Y_t)\,dW_t,\\
\widetilde{b}(t, y) = \lambda U(t, \psi_{t}^{-1}(y)), \
\widetilde{\sigma}(t, y) = I + \nabla U(t, \psi_{t}^{-1}(Y_t))\\
\end{multline*}

{\bf Step 3.} (See \cite{FF}, Proposition 5.2, \cite{FF2}, p.~13--14.)
Taking into account the aforementioned properties of the mappings $\psi_t$ it is not difficult to see
that it  suffices  to prove the existence of a uniformly H\"older-continuous flow
for the transformed equation.
Below we prove only the uniform H\"older-continuity of the desired flow since all other details
(e.g., the proof of its existence) can be found in \cite{FF}.

We have
\begin{equation}\label{eq:transformed_main}
dY_t = \widetilde{b}(t, Y_t)\,dt + \widetilde{\sigma}(t, Y_t)\,dW_t.
\end{equation}
Let us show that for each $a \geqslant 2$ there exists  a constant $C(a, T)$ such that for any $x, y \in \mathbb{R}^{d}$
the following estimate holds:
\begin{equation}
\mathbb{E}\sup\limits_{t \in [0, T]}|Y_{t}^{x} - Y_{t}^{y}|^{a}
\leqslant C(a, T)\bigl(|x - y|^{a} + |x - y|^{a - 1}\bigr),
\end{equation}
In this case the existence of a uniformly H\"older-continuous flow will
 follow from the well-known Kolmogorov continuity theorem.
Following \cite{FF}, \cite{FF2}, let us define an auxiliary process
$$
A_t := \int_{0}^{t}\frac{\|\widetilde{\sigma}(s, Y_{s}^{y}) - \widetilde{\sigma}(s, Y_{s}^{x})\|^{2}}{|Y_{s}^{y} - Y_{s}^{x}|^2}I_{\{Y_{s}^{y} \neq Y_{s}^{x}\}}\,ds
$$
Then (see \cite{FF}, Lemma 4.5) for each $k \in \mathbb{R}$ we have
\begin{equation}\label{eq:exp_estimate}
\mathbb{E}\bigl[e^{k A_T}\bigr] < \infty
\end{equation}
(in the proof of this inequality the Sobolev regularity of $\widetilde{\sigma}$ plays the crucial role).

Let
$$
Z_t := Y_{t}^{y} - Y_{t}^{x}.
$$

Applying It\^o's formula to the process $Z_t$ and the function $f: x \mapsto |x|^{a},$
where $a \geqslant 2,$ we obtain
\begin{multline*}
\frac{1}{a}d|Z_{t}|^{a} = \Bigl<\bigl(\widetilde{b}(t, Y_{t}^{y}) - \widetilde{b}(t, Y_{t}^{x})\bigr)\,dt, Z_{t}^{a - 1}\Bigr> + \\
+ \Bigl<\bigl(\widetilde{\sigma}(t, Y_{t}^{y}) - \widetilde{\sigma}(t, Y_{t}^{x})\bigr)dW_t, Z_{t}^{a - 1}\Bigr> + \\
+ \frac{1}{2}Tr\Bigl([\sigma(t, Y_{t}^{y}) - \sigma(t, Y_{t}^{x})][\sigma(t, Y_{t}^{y}) - \sigma(t, Y_{t}^{x})]^{t} D^{2}f(Z_t)\Bigr)\,dt,\\
\bigl[D^{2}f(Z_t)\bigr]_{i, j} =
\delta_{i,j}|Z_t|^{a - 2} + (a - 2)Z_{t}^{i}Z_{t}^{j}|Z_t|^{a - 4}.\\
\end{multline*}
Using the Lipschitz continuity of $\widetilde{b}$ and the definition of the process $A_t$
we obtain the inequality
$$
d|Z_{t}|^{a} \leqslant C|Z_t|^{a}\,dt + C|Z_t|^{a}\,dA_{t} +
\Bigl<\bigl(\widetilde{\sigma}(t, Y_{t}^{y}) - \widetilde{\sigma}(t, Y_{t}^{x})\bigr)dW_t, Z_{t}^{a - 1}\Bigr>
$$
Let
$$
M_t := \int_{0}^{t}\Bigl<\bigl(\widetilde{\sigma}(t, Y_{t}^{y}) - \widetilde{\sigma}(t, Y_{t}^{x})\bigr)dW_t, Z_{t}^{a - 1}\Bigr>.
$$
Since the coefficient $\widetilde{\sigma}$ is bounded and all moments of
the random variable $|Z_t|$ are finite (see \cite{FF}, Proposition 2.7),
the process $M_t$ is a square-integrable continuous martingale. Then we have
\begin{multline*}
de^{-CA_t}|Z_t|^{a} = -Ce^{-CA_{t}}|Z_{t}|^{a}dA_t + e^{-CA_t}d|Z_t|^{a}\leqslant\\
\leqslant-Ce^{-CA_{t}}|Z_{t}|^{a}dA_t + e^{-CA_t}C|Z_t|^{a}\,dt + e^{-CA_t}C|Z_t|^{a}\,dA_{t} + e^{-CA_t}\,dM_t =\\
=Ce^{-CA_t}|Z_t|^{a}\,dt + e^{-CA_t}\,dM_t\\
\end{multline*}
Consequently, the following estimate holds:
$$
\mathbb{E}e^{-CA_t}|Z_t|^{a} \leqslant |x - y|^{a} + C\int_{0}^{t}\mathbb{E}e^{-CA_t}|Z_t|^{a}\,dt.
$$
Applying Gronwall's inequality we obtain the estimate
$$
\mathbb{E}e^{-CA_t}|Z_t|^{a} \leqslant |x - y|^{a}e^{CT}.
$$
Taking into account H\"older's inequality and the estimate \ref{eq:exp_estimate} we have
$$
\mathbb{E}|Z_t|^{a} = \mathbb{E}e^{CA_t}e^{-CA_t}|Z_t|^a \leqslant
\Bigl[\mathbb{E}e^{2CA_t}\Bigr]^{\frac{1}{2}}\Bigl[\mathbb{E}e^{-2CA_t}|Z_t|^{2a}\Bigr]^{\frac{1}{2}} \leqslant C(a, T)|x - y|^{a}
$$
The next chain of inequalities easily follows from Doob's martingale inequality and
the boundedness of $\widetilde{\sigma}$:
\begin{multline*}
\mathbb{E}\sup\limits_{t \in [0, T]}e^{-2CA_t}|Z_t|^{2a} \leqslant \\
\leqslant 4|Z_0|^{2a} +
4\mathbb{E}\sup\limits_{t \in [0, T]}\Bigl|\int_{0}^{t}Ce^{-CA_s}|Z_s|^{a}\,ds\Bigr|^{2} +
4\mathbb{E}\sup_{t \in [0, T]}\Bigl|\int_{0}^{t}e^{-CA_s}\,dM_s\Bigr|^2 \leqslant\\
\leqslant 4|x - y|^{2a} +
4C^{2}T\mathbb{E}\int_{0}^{T}e^{-2CA_t}|Z_t|^{2a}\,dt + \\
+ 16\mathbb{E}\int_{0}^{T}e^{-2CA_t}\|\widetilde{\sigma}(t, Y_{t}^{y}) - \widetilde{\sigma}(t, Y_{t}^{x})\|^{2}|Z_t|^{2a - 2}\,dt \leqslant\\
\leqslant K|x - y|^{2a} + K\mathbb{E}\int_{0}^{t}e^{-2CA_s}|Z_s|^{2a}\,ds + K\mathbb{E}\int_{0}^{t}e^{-2CA_s}|Z_s|^{2a - 2}\,ds \leqslant \\
\leqslant K(a, T)\bigl(|x - y|^{2a} + |x - y|^{2a - 2}\bigr).\\
\end{multline*}
Therefore,
\begin{multline*}
\mathbb{E}\sup\limits_{t \in [0, T]}|Z_t|^{a} \leqslant \mathbb{E}e^{CA_T}\sup\limits_{t \in [0, T]}e^{-CA_t}|Z_t|^{a}\leqslant\\
\leqslant \Bigl[\mathbb{E}e^{2CA_T}\Bigr]^{\frac{1}{2}} \Bigl[\mathbb{E}\sup\limits_{t \in [0, T]}e^{-2CA_t}|Z_t|^{2a}\Bigr]^{\frac{1}{2}}
\leqslant K'(a, T)\bigl(|x - y|^{a} + |x - y|^{a - 1}\bigr)\\
\end{multline*}
It is now easy to complete the proof.

\section{\sc Main results}
To illustrate the main idea let us prove Davie's theorem for
some (possibly unbounded) drift coefficients $b$ possessing H\"older's continuity
with respect to the space variable. It is worth noting that the
reasoning from \cite{D} can not be directly applied in this case, since they
essentially use the global boundedness of the drift.

\begin{theorem}\label{th:holder_case}
Assume that the coefficient $b$ satisfies the following conditions:
\begin{enumerate}[{1.}]
\item there exists $M_{1} \in L^{q_1}\bigl([0, T], \mathbb{R}\bigr)$ such that
$$|b(t, x)| \leqslant M_{1}(t), \ \ t\in[0, T],\ x\in\mathbb{R}^{d}$$
\item there exists $M_{2} \in L^{q_2}\bigl([0, T], \mathbb{R}\bigr)$ and $\beta > 0$ such that
$$
|b(t, x) - b(t, y)| \leqslant M_2(t)|x - y|^{\beta}, \ \ t\in[0, T],\ x,y\in\mathbb{R}^{d}
$$
\item one has
$$
q_1 \geqslant q_2 > 2, \ \beta > 0, \ \frac{\beta}{p_1} + \frac{1}{p_2} > 1, \ \text{where} \
\frac{1}{p_1} + \frac{1}{q_1} = 1, \ \frac{1}{p_2} + \frac{1}{q_2} = 1.
$$
\end{enumerate}
Then there exist a set $\Omega'$ with $P(\Omega') = 1$ such that for each $\omega \in \Omega'$
the equation \ref{eq:main} has exactly one solution.
\end{theorem}
\begin{proof}
Let $Y_t$ be a solution to the equation \ref{eq:main} for a fixed Brownian trajectory $W$.
Then the following estimate holds:
$$
\max\limits_{t \in [0, T]}|Y_t|
\leqslant |x| + \max_{t\in [0, T]}|W_t| + T^{\frac{1}{p_1}}\|M_1\|_{L^{q_1}[0, T]} =: M(x, W),
$$
so without loss of generality we can assume that $b(t, x) = b(t, x)I_{\{|x| < N\}}$ for some $N > 0$.
Then Proposition \ref{pr:flow} (it is clear that one can take $q_1$ for $q$
and any sufficiently large positive number for $p$) yields
that $P$-almost surely the equation \ref{eq:main} has a H\"older-continuous flow
of solutions which will be denoted by
$X(s, t, x, W), \ s \leqslant t, \ x \in \mathbb{R}^{d}.$

Now let us prove that, for each trajectory $W$ such that there
exists the aforementioned H\"older-continuous flow, the equation
\ref{eq:main} has exactly one solution.
Let us fix $t \in [0, T]$ and define an auxiliary function $f$ by the formula
$$
f(s) = X(s, t, Y_s, W) - X(0, t, x, W), \ s \in [0, t].
$$
From the definition of $f$ and the H\"older-continuity of the flow $X(s, t, x, W)$ we obtain that for all
$u, r: 0 \leqslant u \leqslant r \leqslant t$ one has
\begin{multline*}
|f(r) - f(u)| = |X(r, t, Y_r, W) - X(u, t, Y_u, W)| = \\
= |X(r, t, Y_r, W) - X(r, t, X(u, r, Y_u, W), W)| \leqslant \\
\leqslant C(\alpha, T, M(x, W), \omega)|Y_r - X(u, r, Y_u, W)|^{\alpha}. \\
\end{multline*}
Let us estimate $|Y_r - X(u, r, Y_u, W)|$. It is clear that we have the following trivial bound:
\begin{multline*}
|Y_r -  X(u, r, Y_u, W)| \leqslant \int_{u}^{r}|b(s, Y_s) - b(s, X(u, s, Y_u, W))|\,ds \leqslant\\
\leqslant 2\int_{u}^{r}M_{1}(s)\,ds \leqslant 2\|M_1\|_{L^{q_1}[0, T]}|r - u|^{\frac{1}{p_1}}\\
\end{multline*}
The previous estimate can be improved if we take into account the H\"older-continuity of the coefficient $b$:
\begin{multline*}
|Y_r -  X(u, r, Y_u, W)| \leqslant \int_{u}^{r}|b(s, Y_s) - b(s, X(u, s, Y_u, W))|\,ds \leqslant\\
\leqslant \int_{u}^{r}M_{2}(s)|Y_s - X(u, s, Y_u, W)|^{\beta}\,ds
\leqslant K'\int_{u}^{r}M_{2}(s)|r - u|^{\frac{\beta}{p_1}}\,ds \leqslant\\
\leqslant K'\|M_{2}\|_{L^{q_2}[0, T]}|r - u|^{\frac{\beta}{p_1} + \frac{1}{p_2}}.\\
\end{multline*}
Let us pick $\alpha \in (0, 1)$ such that $\frac{\alpha\beta}{p_1} + \frac{\alpha}{p_2} = 1 + \delta, \ \delta > 0$.
Then we have
$$
|f(r) - f(u)| \leqslant C(\alpha, T, M(x, W), \omega)|r - u|^{1 + \delta}.
$$
Consequently, $f \equiv 0$ (here we have also used the fact that $f(0) = 0$,
 which is clear from the definition of $f$).
Finally, $Y_t = X(0, t, x, W)$ and we obtain the desired assertion, since $t \in [0, T]$ was arbitrary.
\end{proof}

Now we show how to prove the original  result of Davie (his Theorem \ref{th:borel_case})
in the case where $b$ is just Borel measurable.
Similarly to the proof of Theorem \ref{th:holder_case}, it is readily seen
that without loss of generality we can assume that
$b(t, x) = b(t, x)I_{\{|x| < N\}}$ and $\|b\|_{\infty} \leqslant 1$.
In this case for each $\alpha \in (0, 1)$ the equation \ref{eq:main} $P$-almost surely
possesses a H\"older-continuous flow
of solutions that will be denoted by $X(s, t, x, W)$.
The main aim of the reasoning below is to find a substitute
for the H\"older condition on the coefficient $b$ that would allow us
to repeat  the proof of Theorem \ref{th:holder_case} with minor changes.

Below we will need the following set of functions:
\begin{multline*}
Lip_{N}\bigl([r, u], \mathbb{R}^{d}\bigr) := \\
:= \bigl\{h \in C\bigl([r, u], \mathbb{R}^{d}\bigr) \ \mid \ |h(t) - h(s)| \leqslant |t - s| \ s,t \in [r, u], \ \max\limits_{s \in [r, u]}|h(s)| \leqslant N \bigr\}
\end{multline*}
with the uniform metric $\varrho(h_1, h_2) = \|h_1 - h_2\|_{\infty}$.
\begin{lemma}\label{le:entropy}
There exist constants $C, \gamma > 0$ such that for all $N, \varepsilon > 0$
 the set $Lip_{N}\bigl([r, u], \mathbb{R}^{d}\bigr)$
contains an $\varepsilon$-net $\mathcal{N}_{\varepsilon}$ with no more than
$$C\Bigl(\frac{N}{\varepsilon}\Bigr)^{d}\exp\Bigl(\gamma\frac{u - r}{\varepsilon}\Bigr)$$
elements.
\end{lemma}
\begin{proof}
This estimate can be easily obtained from formula 7 in Section 2 of \cite{KT}.
\end{proof}

Now let us temporarily fix $N > 0$ and $r, u \in [0, T]$ such that $l = u - r \leqslant \frac{1}{2}$.

Let
$$
\varphi(h, W) := \int_{r}^{u}b(s, W_s + h(s))\, ds.
$$

\begin{lemma}\label{le:main_borel_estimate}
There exist constants $C, \zeta > 0$, independent of $l = u - r$,
 a countable dense subset $\mathcal{N}$ in $Lip_{N}\bigl([r, u], \mathbb{R}^{d}\bigr)$,
independent of $b$, and a set $\Omega'$ such that
$$
P(\Omega \setminus \Omega') \leqslant C\exp\Bigl(-l^{-\zeta}\Bigr)
$$
and for any $h_1, h_2 \in \mathcal{N}$ with $\|h_1 - h_2\|_{\infty} \leqslant 3l$ and $W \in \Omega'$
the following inequality holds:
$$
|\varphi(h_1, W) - \varphi(h_2, W)| \leqslant Cl^{\frac{4}{3}}.
$$
\end{lemma}
\begin{proof}
Let $\alpha$ and $\gamma$  be positive constants from Corollary \ref{cor:refined_estimate}
and Lemma \ref{le:entropy}, respectively.
Let us define sequences $\{\varepsilon_{k}\}_{k \geqslant 0}, \ \{\lambda_{k}\}_{k \geqslant 0}$ as follows:
$$
\varepsilon_{k} = l^{1 + \frac{1}{4}k}, \ \lambda_{k} = \mu l^{-\frac{1}{6} - \frac{1}{6}k}, \ \text{where} \ \mu^2 = \frac{\gamma + 1}{\alpha}.
$$
Let  $\pi_{k}$ denote the mapping that sends a function from $Lip_{N}\bigl([r, u], \mathbb{R}^{d}\bigr)$ to the
nearest element in the $\varepsilon_{k}$-net $\mathcal{N}_{\varepsilon_{k}}$.
For each $g_{k + 1} \in \mathcal{N}_{\varepsilon_{k + 1}}$ let
$$
\Omega_{g_{k + 1}} := \bigl\{W: |\varphi(g_{k + 1}, W) - \varphi(\pi_{k}(g_{k + 1}), W)| \geqslant l^{\frac{1}{2}}\varepsilon_{k}\lambda_{k}\bigr\},
$$
$$
\Omega_{k + 1} := \bigcup_{g_{k + 1} \in \mathcal{N}_{\varepsilon_{k + 1}}}\Omega_{g_{k + 1}}.
$$
Let $\theta$ be a positive constant, below we will explain how $\theta$ should be chosen.
Now for each pair for functions $f_1, f_2 \in \mathcal{N}_{\varepsilon_{0}}$ with
$$
\|f_1 - f_2\|_{\infty} \leqslant \theta\varepsilon_0
$$
we introduce the the sets
$$
\Omega_{f_1, f_2} := \bigl\{W: |\varphi(f_1, W) - \varphi(f_2, W)| \geqslant l^{\frac{1}{2}}\theta\varepsilon_{0}\lambda_0\bigr\},
$$
$$
\Omega_0 := \bigcup_{f_1, f_2: \ \|f_1 - f_2\|_{\infty} \leqslant \theta\varepsilon_0} \Omega_{f_1, f_2}.
$$
One can observe that for any $g_{k + 1} \in \mathcal{N}_{\varepsilon_{k + 1}}$ we have
$$
\|g_{k + 1} - \pi_{k}(g_{k + 1})\|_{\infty} \leqslant \varepsilon_{k}.
$$
Applying Corollary \ref{cor:refined_estimate} we obtain the following inequalities:
$$
P(\Omega_{k + 1}) \leqslant \sum_{g_{k + 1} \in \mathcal{N}_{\varepsilon_{k + 1}}}P\bigl(\Omega_{g_{k + 1}}\bigr) \leqslant
C\Bigl(\frac{N}{\varepsilon_{k + 1}}\Bigr)^{d}\exp{\Bigl(\frac{\gamma l}{\varepsilon_{k + 1}} - \alpha\lambda_{k}^2\Bigr)}
$$
$$
P(\Omega_0) \leqslant \sum_{f_1, f_2} P\bigl(\Omega_{f_1, f_2}\bigr) \leqslant
C^{2}\Bigl(\frac{N}{l}\Bigr)^{2d}\exp{\Bigl(\gamma - \alpha \mu^2 l^{-\frac{1}{3}}\Bigr)}
$$
Since
$$
\frac{N^d}{\varepsilon_{k + 1}^d} = N^{d}l^{-\frac{5}{4}d - dk} \leqslant N^{d}l^{-\frac{5}{4}d(k + 1)},
$$
$$
\frac{\gamma l}{\varepsilon_{k + 1}} - \alpha \lambda_{k}^{2} =
\gamma l^{-\frac{1}{4} - \frac{1}{4}k} - \alpha \mu^{2}l^{-\frac{1}{3} - \frac{1}{3}k} =
\gamma l^{-\frac{1}{4} - \frac{1}{4}k} - (\gamma + 1)l^{-\frac{1}{3} - \frac{1}{3}k}
\leqslant
-l^{-\frac{1}{3}(k + 1)},
$$
it can be easily verified that there exist positive constants $\zeta$ and $C$ such that for any
 $k\ge 0$ the following inequalities hold:
$$
P\bigl(\Omega_{k + 1}\bigr) \leqslant C\exp{\bigl(-l^{-\zeta(k + 1)}\bigr)}, \ \ P\bigl(\Omega_{0}\bigr) \leqslant C\exp{\bigl(-l^{-\zeta}\bigr)}.
$$
Let
$$
\Omega' := \Omega \setminus \bigcup_{k = 0}^{\infty}\Omega_{k}, \ \ \mathcal{N} := \bigcup_{k = 0}^{\infty}\mathcal{N}_{\varepsilon_k}.
$$
Taking into account the reasoning above we have
$$
P\bigl(\Omega \setminus \Omega'\bigr) \leqslant C(T, N)\exp{\bigl(-l^{-\zeta}\bigr)},
$$
$$
\mathcal{N} \text{ is a dense subset of } \ Lip_{N}\bigl([r, u], \mathbb{R}^{d}\bigr).
$$
Let $W$ be an arbitrary trajectory in $\Omega'$ and let $h_1, h_2$ be two functions in $\mathcal{N}$ with
$\|h_1 - h_2\|_{\infty} \leqslant 3l.$
Let us assume that $h_1 \in \mathcal{N}_{\varepsilon_{k_1}}, h_2 \in \mathcal{N}_{\varepsilon_{k_2}}$.
Then we can construct two sequences of functions:
$$
h_{1,k_1} = h_1, \, h_{1, k_1 - 1} = \pi_{k_1 - 1}(h_{1,0}), \, \pi_{k_1 - 2}(h_{1, k_1 - 1}), \, \ldots, \, h_{1, 0} = \pi_{0}(h_{1,1}),
$$
$$
h_{2,k_2} = h_2, \, h_{2, k_2 - 1} = \pi_{k_2 - 1}(h_{2,0}), \, \pi_{k_2 - 2}(h_{2, k_2 - 1}), \, \ldots, \, h_{2, 0} = \pi_{0}(h_{2,1}).
$$
It is not difficult to show that due to our choice of $W$ we can
find a positive number $K$ (which does not depend on $\theta$) such that
the following inequalities hold:
$$
\|h_1 - h_{1,0}\|_{\infty} \leqslant Kl, \ \ \|h_2 - h_{2,0}\|_{\infty} \leqslant Kl.
$$
Consequently, taking $2K + 3$ for $\theta$, we obtain
$$
\|h_{1,0} - h_{2, 0}\|_{\infty} \leqslant \theta l,
$$
in particular, the set $\Omega_0$ contains $\Omega_{h_{1, 0}, h_{2, 0}}$.

Now since $W \in \Omega'$ and
$$
l^{\frac{1}{2}}\varepsilon_{k}\lambda_k = \mu l^{\frac{4}{3} + \frac{1}{12}k},
$$
we conclude that there exists a positive constant $C = C(N, T)$ such that the following estimate holds:
$$
|\varphi(h_1, W) - \varphi(h_2, W)| \leqslant Cl^{\frac{4}{3}}.
$$
\end{proof}

\begin{lemma}\label{le:open_set_indicator_estimate}
For any $\varepsilon, N > 0$ there exists $\delta > 0$ such that
for each open set $U \subset [0, 1]\times \mathbb{R}^{n}$ with $\lambda(U) < \delta$
there is a Borel set of Brownian trajectories $\Omega_{\varepsilon}$ with
$P\bigl(\Omega_{\varepsilon}\bigr) \geqslant 1 - \varepsilon$ such that for any
$W \in \Omega_{\varepsilon}, \ h \in Lip_{N}\bigl([0, 1], \mathbb{R}^{d}\bigr)$ the following inequality holds:
$$
\int_{0}^{1}I_{U}(s, W_s + h(s))\, ds \leqslant \varepsilon.
$$
\end{lemma}
\begin{proof}
Assume we are given $\varepsilon, N > 0$.
Let us choose $l > 0$ such that
$$
\frac{2}{l}C\exp{\bigl(-l^{-\zeta}\bigr)} \leqslant \frac{\varepsilon}{2}, \ \
\frac{2}{l}Cl^{\frac{4}{3}} \leqslant \frac{\varepsilon}{2},
$$
where
$C, \zeta$ are positive constants from Lemma \ref{le:main_borel_estimate}.
Next let us split the interval $[0, 1]$ into a collection of closed subintervals $\Delta_1, \ldots, \Delta_M$
of length less than $l$, $M \leqslant \frac{2}{l}$.
Applying Lemma \ref{le:main_borel_estimate} to each interval $\Delta_k$
we can find countable sets $\mathcal{N}_{1}, \ldots, \mathcal{N}_{M}$
(here we also use the fact that these subsets do not depend on $b$,
 see Lemma \ref{le:main_borel_estimate}).
Now in each $\mathcal{N}_{s}$ we take a finite $3l$-net that will be denoted by $\mathcal{N}'_{s}$.
Let us pick $\delta > 0$ such that for each open set $U$ with
$\lambda(U) \leqslant \delta$ there exists a set $\Omega'$ such that
$P\bigl(\Omega'\bigr)\geqslant 1 - \frac{\varepsilon}{2}$ and
for any $W \in \Omega',\ h \in \mathcal{N}'_{s}$ one has
$$
\int_{\Delta_s}I_{U}(s, W_s + h(s))\, ds \leqslant \frac{l\varepsilon}{4}
$$
(such $\delta$ obviously exists).
Let us prove that this $\delta$ satisfies the  conditions stated above.

Let us fix an open set $U$ with $\lambda(U) \leqslant \delta$.
Applying Lemma \ref{le:main_borel_estimate} for each $s$ one can find a set $\Omega_s$ with
$$
P\bigl(\Omega \setminus \Omega_s\bigr) \leqslant C\exp\bigl(-l^{-\zeta}\bigr),
$$
such that for any $h_1, h_2 \in \mathcal{N}_{s}$ with
$\|h_1 - h_2\|_{\infty} \leqslant 3l$ and $W \in \Omega_s$
the following inequality holds:
$$
\Bigl|\int_{\Delta_{s}}I_{U}(s, W_s + h_{1}(s))\,ds - \int_{\Delta_{s}}I_{U}(s, W_s + h_{2}(s))\,ds \Bigr| \leqslant Cl^{\frac{4}{3}}.
$$
Let
$$
\Omega_{\varepsilon} := \Omega' \cap \bigcap_{s=1}^{M}\Omega_s.
$$
Let us observe that
$$
P\bigl(\Omega_{\varepsilon}\bigr) \geqslant 1 - \varepsilon,
$$
and for each $h_s \in \mathcal{N}_{s}$
$$
\int_{\Delta_s}I_{U}(s, W_s + h(s))\, ds \leqslant \frac{l\varepsilon}{4}.
$$
Since $U$ is open, applying Fatou's lemma we conclude that the previous inequality is true for all
$h \in Lip_{N}\bigl(\Delta_s, \mathbb{R}^{d}\bigr)$. It is now  trivial to
complete the proof.
\end{proof}

\begin{lemma}\label{le:pass_to_limit}
Let $b: [0, T]\times \mathbb{R}^{d}\mapsto\mathbb{R}^{d}$ be a bounded Borel measurable
mapping with $\|b\|_{\infty} \leqslant 1$.
Then there exists a set $\Omega'$ with $P\bigl(\Omega'\bigr) = 1$ such that for each $W \in \Omega'$
and each sequence
of functions $\{h_k\} \subset Lip_{N}\bigl([0, 1], \mathbb{R}^{d}\bigr)$ pointwise converging
to a function $h$ the following equality holds:
$$
\lim_{k \to \infty} \int_{0}^{1}b(s, W_s + h_k(s))\, ds = \int_{0}^{1}b(s, W_s + h(s))\, ds
$$
\end{lemma}
\begin{proof}
Applying Lemma \ref{le:open_set_indicator_estimate} for each $\varepsilon_{n} = \frac{1}{2^n}$
we can find the corresponding $\delta_{n} > 0$.
Next, applying Lusin's theorem for each $n$ one can find a function
$b_{n} \in C_{b}\bigl([0, 1]\times \mathbb{R}^{d}, \mathbb{R}^{d}\bigr)$
and an open set $U_{n} \subset [0, 1]\times \mathbb{R}^{d}$ such that
$$
\|b_{n}\|_{\infty} \leqslant 1, \ \ \lambda(U_n) \leqslant \delta_{n}, \ \  b_{n}(t, x) = b(t, x) \ \text{for all} \ (t, x) \notin U_n.
$$
Then there exists a set $\Omega_n$ with the following properties:
$$
P\bigl(\Omega_n) \geqslant 1 - \varepsilon_{n}
$$
and for any $W \in \Omega_{n}, \ h \in Lip_{N}\bigl([0, 1], \mathbb{R}^{d}\bigr)$
$$
\int_{0}^{1}I_{U}(s, W_s + h(s))\,ds \leqslant \varepsilon_n.
$$
Next we observe that for any $n$
\begin{multline*}
\int_{0}^{1}b_{n}(s, W_s + h(s))\,ds - 2\int_{0}^{1}I_{U}(s, W_s + h(s))\,ds \leqslant\\
\leqslant \int_{0}^{1}b(s, W_s + h(s))\,ds \leqslant\\
\leqslant \int_{0}^{1}b_{n}(s, W_s + h(s))\,ds + 2\int_{0}^{1}I_{U}(s, W_s + h(s))\,ds. \\
\end{multline*}
Therefore, for each $W \in \Omega_{n}$ and each sequence of functions
$\{h_k\} \subset Lip_{N}\bigl([0, 1], \mathbb{R}^{d}\bigr)$ pointwise converging to $h$,
 the following inequalities hold:
\begin{multline*}
\int_{0}^{1}b(s, W_s + h(s))\, ds - 4\varepsilon_n \leqslant
\int_{0}^{1}b_{n}(s, W_s + h(s))\, ds - 2\varepsilon_n \leqslant \\
\leqslant \liminf_{k \to \infty} \int_{0}^{1}b_{n}(s, W_s + h_k(s))\, ds - 2\varepsilon_n \leqslant
\liminf_{k \to \infty} \int_{0}^{1}b(s, W_s + h_k(s))\,ds\\
\int_{0}^{1}b(s, W_s + h(s))\, ds + 4\varepsilon_n \geqslant
\int_{0}^{1}b_{n}(s, W_s + h(s))\, ds + 2\varepsilon_n \geqslant \\
\geqslant \limsup_{k \to \infty} \int_{0}^{1}b_{n}(s, W_s + h_k(s))\, ds + 2\varepsilon_n \geqslant
\limsup_{k \to \infty} \int_{0}^{1}b(s, W_s + h_k(s))\,ds\\
\end{multline*}
Let
$$
\Omega' := \liminf_{n \to \infty}\Omega_n = \bigcup_{m = 1}^{\infty}\bigcap_{n = m}^{\infty}\Omega_n.
$$
Since
$$
P\bigl(\Omega_n) \geqslant 1 - \varepsilon_{n} \
\text{and} \ \sum\limits_{n = 1}^{\infty}\varepsilon_n < \infty,
$$
 by the Borel--Cantelli lemma $P\bigl(\Omega'\bigr) = 1$.
 It is now trivial to complete the proof.
\end{proof}

\begin{lemma}\label{le:refined_main_borel_estimate}
There exist constants $C, \zeta > 0$, independent of $l = u - r$, and a set $\Omega'$ such that
$$
P(\Omega \setminus \Omega') \leqslant C\exp\Bigl(-l^{-\zeta}\Bigr)
$$
and for any $h_1, h_2 \in \mathcal{N}$ with $\|h_1 - h_2\|_{\infty} \leqslant 4l$, $W \in \Omega'$
the following inequality holds:
$$
|\varphi(h_1, W) - \varphi(h_2, W)| \leqslant Cl^{\frac{4}{3}}.
$$
\end{lemma}
\begin{proof}
This assertion  follows directly from Lemma \ref{le:main_borel_estimate} and Lemma \ref{le:pass_to_limit}.
\end{proof}

We can now proceed to the proof of Theorem \ref{th:borel_case}.

\begin{proof}
Let us fix a positive number $N$.
Let $C, \zeta$ be constants found in Lemma \ref{le:refined_main_borel_estimate}.
For each $k$ we split the interval $[0, 1]$ into $M = 2^k$ closed subintervals
$$
\Bigl[0, \frac{1}{M}\Bigr], \ldots, \Bigl[\frac{M-1}{M}, M\Bigr].
$$
Applying Lemma \ref{le:refined_main_borel_estimate} to each interval
$\bigl[\frac{i}{M}, \frac{i + 1}{M}\bigr]$ we can find the corresponding sets $\Omega_{k, i}$.
Let
$$
\Omega_k := \bigcap_{i = 0}^{M - 1}\Omega_{k, i}.
$$
With the help of the Borel--Cantelli lemma it is easy to show that the set
$$
\Omega' := \liminf_{k \to \infty}\Omega_k = \bigcup_{K = 1}^{\infty}\bigcap_{k = K}^{\infty}\Omega_k
$$
has probability $1$.
Removing, if necessary,  a set of zero probability from $\Omega'$,
 we can assume that for each
$W \in \Omega'$ there exists a H\"older-continuous flow ensured by Proposition \ref{pr:flow}.
Let us show that for each $W \in \Omega'$ such that
$$
|x| + \max\limits_{t \in [0, 1]}|W_t| + 1 \leqslant N,
$$
the equation \ref{eq:main} has a unique solution.
Indeed, let $Y_t$ be a solution to the equation \ref{eq:main}.
It is not difficult to see that $|Y_t| \leqslant N$ for each $t \in [0, 1]$.
Due to our choice of $\Omega'$ there exists $K = K(\omega)$ such that for
all $k \geqslant K$ the Brownian trajectory $W$ belongs to $\Omega_k$.
Let
$$
M' = 2^{k'}, \ \ r = \frac{i}{M'}, \ \ \text{where} \ k' \geqslant K.
$$
Let us define an auxiliary function $f$ on the interval $[0, r]$ by the following formula:
$$
f(t) := X(x, 0, r, W) - X(Y_t, t, r, W).
$$
We observe that for any $s \leqslant t$, by to the definition of a flow we have
\begin{multline*}
f(t) - f(s) = -X\bigl(Y_t, t, r, W\bigr) + X\bigl(Y_s, s, r, W\bigr) = \\
= -X\bigl(Y_t, t, r, W\bigr) + X\bigl(X(Y_s, s, t, W), r, W\bigr).\\
\end{multline*}
Hence there exists a positive constant $C = C(N, W)$ such that
$$
|f(t) - f(s)| \leqslant C|Y_t - X(Y_s, s, t, W)|^{\frac{4}{5}}.
$$
The difference $Y_t - X(Y_s, s, t, W)$ can be represented as follows:
\begin{multline*}
Y_t - X(Y_s, s, t, W) = \\
=\int_{s}^{t}b\Bigl(u, Y_s + W_u - W_s + \int_{s}^{u}b(r, Y_r)\,dr\Bigr)\,du - \\
\int_{s}^{t}b\Bigl(u, Y_s + W_u - W_s + \int_{s}^{u}b(r, X_r)\,dr\Bigr)\,du = \\
= \int_{s}^{t}b\bigl(u, W_u + h_1(u)\bigr)\,du - \int_{s}^{t}b\bigl(u, W_u + h_2(u)\bigr)\,du,
\end{multline*}
where
$$
h_1(u) = Y_s - W_s + \int_{s}^{u}b(r, Y_r)\,dr, \ \ h_2(u) = Y_s - W_s + \int_{s}^{u}b(r, X_r)\,dr.
$$
Let $k \geqslant k'$ и $M = 2^{k}$.
If we take $s, t$ of the form $\frac{i}{M}$ and $\frac{i + 1}{M}$, respectively,
then we obtain the following estimate:
$$
\Bigl|f\Bigl(\frac{i + 1}{M}\Bigr) - f\Bigl(\frac{i}{M}\Bigr)\Bigr| \leqslant \Bigl(\frac{C}{M^{\frac{4}{3}}}\Bigr)^{\frac{4}{5}},
$$
and consequently
$$
|f(r)| \leqslant \frac{C}{M^{\frac{1}{15}}}.
$$
Due to the arbitrariness of $k$ we conclude
$$
f(r) = X(x, 0, r, W) - Y_r = 0.
$$
Since $r$ was an arbitrary dyadic number in $[0, 1]$  with  a sufficiently large denominator,
the continuity of $Y_t$ and $X(x, 0, t, W)$ implies the equality $Y_t = X(x, 0, t, W)$ for each $t \in [0, 1]$.
The proof is complete.
\end{proof}

\centerline{{\bf\it Acknowledgment}}

I would like to thank V.I. Bogachev for fruitful discussions and comments.
A part of this work was done during a visit to the Mathematical Institute of Burgundy.
I would like to thank Shizan Fang for his hospitality and useful remarks.
This work has been supported by the RFBR project 12-01-33009.


\begin{thebibliography}{}

\bibitem{D}
Davie A.M. Uniqueness of solutions of stochastic differential equations.
International Mathematics Research Notices, 2007, V.~2007.

\bibitem{B} Bogachev V.I.
Differentiable measures and the Malliavin calculus.
Amer. Math. Soc., Rhode Island, Providence, 2010.

\bibitem{FF}
Fedrizzi E., Flandoli F.
H\"older flow and differentiability for SDEs with nonregular drift.
Stoch. Anal. Appl., 2013, V. 31, N4, P. 708--736.

\bibitem{FF2}
Fedrizzi E., Flandoli F.
Pathwise uniqueness and continuous dependence for SDEs with nonregular drift.
arXiv preprint arXiv:1004.3485, 2010.

\bibitem{VK}
Van Kampen E.R. Remarks on systems of ordinary differential equations.
Amer. J. Math., 1937, V. 59, N1, P. 144--152.

\bibitem{F}
Flandoli F. Regularizing properties of Brownian paths and a result of Davie.
Stochastics and Dynamics, 2011, V. 11, N02n03,  P. 323--331.

\bibitem{FPS}
F\"ollmer H., Protter P., Shiryaev A.N.
 Quadratic covariation and an extension of Ito's formula.
 Bernoulli, 1995, V.1, N1-2, P. 149--169.

\bibitem{KT}
Kolmogorov A.N., Tikhomirov V. M.
 $\varepsilon$-entropy and $\varepsilon$-capacity of sets in function spaces.
Uspekhi Matem. Nauk, 1959, V. 14, N2, P. 3--86 (in Russian). English translation: Amer. Math. Soc. Transl. Ser. 2, 1961, V. 17, P. 277--364

\bibitem{KR}
Krylov N.V., R\"ockner M. Strong solutions of stochastic equations with
singular time dependent drift. Probab. Theory Related Fields, 2005, V. 131, N2, P. 154--196.

\end{thebibliography}
\end{document}